\documentclass[submission,copyright,creativecommons]{eptcs}
\providecommand{\event}{ACT 2026} %

\usepackage{iftex}

\ifpdf
  \usepackage{underscore}         %
  \usepackage[T1]{fontenc}        %
\else
  \usepackage{breakurl}           %
\fi

\title{Metacat\\\large{}a categorical framework for formal systems}
\author{Paul Wilson
\institute{Hellas AI}
\email{paul@hellas.ai}
\email{paul@statusfailed.com}
}

\newcommand{\titlerunning}{Metacat}
\newcommand{\authorrunning}{Paul Wilson}

\hypersetup{
  bookmarksnumbered,
  pdftitle    = {\titlerunning},
  pdfauthor   = {\authorrunning},
  pdfsubject  = {EPTCS},               %
}

\usepackage{amsmath, amssymb, amsthm}
\usepackage{mathrsfs} %
\usepackage{hyperref}
\usepackage{graphicx}
\usepackage{float}
\usepackage{xcolor}
\usepackage{enumitem}
\usepackage{stmaryrd}

\newtheorem{observation}{Remark}[section]
\newtheorem{theorem}[observation]{Theorem}
\newtheorem{definition}[observation]{Definition}

\newtheorem{remark}[observation]{Remark}

\newcommand{\cp}[0]{\ensuremath{\fatsemi}} %

\newcommand{\tensor}[0]{\ensuremath{\otimes}}

\newcommand{\imp}[0]{\ensuremath{\Rightarrow}}

\newcommand{\cat}[1]{\ensuremath{\mathcal{#1}}}

\begin{document}
\maketitle

\begin{abstract}
  \label{section:abstract}
  We present a categorical framework for formal systems in which inference rules
with $m$ metavariables over a category of syntax $\mathscr{S}$, taken to be a
cartesian PROP, are represented
by operations of arity $k \to n$ equipped with spans
$k \leftarrow m \to n$
in $\mathscr{S}$, encoding the hypotheses and conclusions in a common
metavariable context.
Composition is by substitution of metavariables, which is
the sole primitive operation, as in Metamath.

Proofs in this setting form a symmetric monoidal category whose monoidal
structure encodes the combination and reuse of hypotheses. This structure
admits a proof-checking algorithm; we provide an open-source implementation
together with a surface syntax for defining formal systems. As a demonstration,
we encode the formulae and inference rules of first-order logic in Metacat, and
give axioms and representative derivations as examples.

\end{abstract}

\section{Introduction}
\label{section:introduction}

Metacat is a framework for presenting formal systems in a way that makes the
compositional structure of proofs explicit.
Inference rules are modeled as transformations of \emph{input syntax} to
\emph{output syntax} by substitution of metavariables.
The resulting proofs are not merely derivation trees written in a different
notation: they assemble into a symmetric monoidal category, so that independent
subproofs can be juxtaposed and reused in a principled way.
This categorical structure gives a semantics for proof assembly and suggests a
natural visual language in terms of string diagrams.
For example, a proof of $\vdash P \imp P$ using the Metamath axioms
$\mathsf{ax\text{-}1}$ (Simp), $\mathsf{ax\text{-}2}$ (Frege), and $\mathsf{ax\text{-}mp}$ (modus
ponens) is depicted as a string diagram below.
\begin{figure}[H]
\centering
\includegraphics[width=\linewidth]{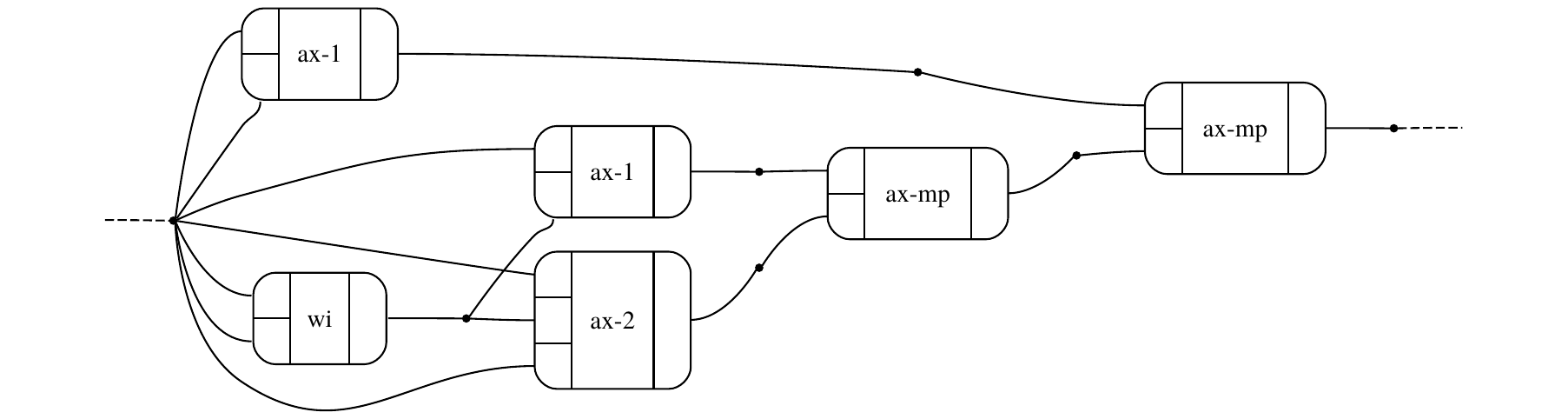}
\caption{A string-diagrammatic proof of $\vdash P \imp P$.}
\label{figure:diagram-of-id}
\end{figure}

This perspective is motivated by desire for a small, core proof kernel.
In Metamath and Metamath Zero for example, the essential operation is
substitution of formulae for metavariables, with very little logic built into
the kernel itself.
We adopt that discipline here.
A minimal substitution-based kernel is attractive for two reasons.
First, as the Metamath Zero thesis makes clear~\cite{mm0}, it is easier to trust and
verify a small checker than a large proof engine with many primitive notions
baked in.
Second, keeping the kernel agnostic about a particular logic makes the
framework flexible: the ambient machinery does not presuppose classical logic,
type theory, or any other foundational choice, but instead lets a concrete
formal system be specified by its syntax and rules.
The categorical presentation clarifies this idea by treating substitution as
composition and by treating collections of hypotheses as monoidal structure.

The categorical viewpoint also addresses two practical issues that arise in
formal syntax.
On the proof side, it provides formally justified tools for visualisation
rather than relying on ad hoc proof drawings.
On the syntax side, it allows for the combinatorial representation of proofs as
cospans of hypergraphs~\cite{sdrt1, sdrt2, sdrt3}, in which connectivity is
primary and concerns about variable names or $\alpha$-equivalence can be
handled structurally.
Our goal is therefore not to design a diagrammatic notation for first-order
logic itself, in the style of systems where the diagrams are the
formulas\cite{dafol}.
Rather, we use categorical and string-diagrammatic methods to expose the
structure of proofs in a general substitutional framework.

The main contribution of this paper is a categorical account of formal proofs
in which proof checking is reduced to evaluation in a symmetric monoidal
category of partial functions.
Proof rules are represented syntactically by spans over a category of syntax,
derivations form a symmetric monoidal category $\cat{P}$, and there is an
operational interpretation as partial functions
\[
  \llbracket - \rrbracket : \cat{P} \to \mathbf{Pfn}
\]
such that a derivation is valid exactly when the corresponding partial map is
defined at the claimed boundary.
This interpretation is realized by Metacat, an implementation of the proof
checker which verifies proofs by evaluating the associated open hypergraph.

\subsection{Related Work}
Metacat is closest in spirit to Metamath and Metamath Zero~\cite{metamath,mm0}.
It shares their emphasis on a tiny trusted core and on explicit substitution,
but recasts that discipline in a language where compositional structure is
first-class.
One way to view the project is as a string-diagrammatic analogue of Metamath
Zero: not a competing foundation, but a reformulation that makes the algebra of
proofs easier to see and manipulate with the tools of category theory.
The difference is one of presentation: rather than treating proofs primarily as
trees of rule applications, we make their compositional structure explicit in a
symmetric monoidal setting.
In this respect the project also draws on categorical work on open graphs,
hypergraphs, and diagrammatic rewriting~\cite{sdrt1,sdrt2,sdrt3}, which provide
natural combinatorial models for syntax with sharing.
By contrast with systems such as Chyp~\cite{chyp}, our proofs are not obtained
by rewriting string diagrams. Rather, derivations are represented explicitly as
diagrams and checked by an external operational semantics.
It also contrasts with proof assistants based on richer built-in logics, such
as Lean and Agda~\cite{deMoura2021lean4,norell2007agda}, where considerably
more logical structure is part of the trusted core.
At the same time, our aim differs from diagrammatic logics such as
diagrammatic first-order logic~\cite{dafol}: the diagrams here are not the
formulas themselves, but rather a representation of the syntax and composition
of proofs.
Finally, unlike proof-net style representations in linear logic
~\cite{girard1987linear,curien2005linear2,mellies2009categorical}, we do not
try to build global correctness directly into the proof syntax; as in
Metamath, the syntax of derivations is simple, and correctness is imposed by a
separate checking semantics.

\subsection{Synopsis}
The remainder of the paper is structured as follows.
Section~\ref{section:syntax-and-proof} introduces the categories of syntax and
proofs used throughout the paper, and
Section~\ref{section:proof-checking} gives the corresponding proof-checking
semantics and operational interpretation.
Section~\ref{section:fol} then illustrates the framework on a fragment of
first-order logic in the style of Metamath.
Finally, Section~\ref{section:conclusions} closes with brief remarks on
implementation and future work.

\section{Syntax and Proof}
\label{section:syntax-and-proof}

This section introduces the basic technical objects used throughout the paper.
We begin with a motivating example from propositional logic, then explain how
syntax with metavariables is represented categorically, and finally describe
how axioms and inference rules generate a category of proofs over this syntax.
The main point is that the interface of a proof already has symmetric monoidal
shape: a proof takes a finite family of hypotheses to a finite family of
conclusions, and larger proofs are built by juxtaposing and composing smaller
ones.

\subsection{A motivating example}

Consider the inference rule of modus ponens. Informally, it consumes two
hypotheses, namely a proposition $A$ and an implication $A \imp B$, and
produces the conclusion $B$.
As a morphism in a symmetric monoidal category, we might think of it as having
the type $A \tensor (A \imp B) \to B$.
This already illustrates two features that will persist in the general theory.
First, inference rules are naturally many-input, many-output operations.
Second, they are parameterised by \emph{metavariables}: the two occurrences of
$A$ and the two occurrences of $B$ are not independent pieces of syntax, but
must be understood as referring to the same placeholders throughout the rule.

This dependence on shared metavariables is exactly where naive representations
of syntax become awkward. If formulas are represented only as strings or parse
trees with named variables, then composition of rules requires bookkeeping for
substitution, matching of repeated variables, and avoidance of accidental name
capture. Those are familiar problems, but they obscure the compositional
content we wish to expose. Our aim is instead to use a representation in which
sharing is explicit and structural.

\subsection{Categories of syntax}

We begin by explicitly formalising the notion of syntax used in this paper.
\begin{definition}\label{def:syntax-category}
  A \emph{category of syntax} $\cat{S}$ is a free cartesian PROP generated by a
  set of operations $\Sigma_1$.
\end{definition}
Here $\Sigma_1$ should be thought of as the collection of term formers of the
language. An operation $n \to 1$ in $\Sigma_1$ is therefore an $n$-ary term
former. In propositional logic, for example, one may take a binary operation
$\imp : 2 \to 1$
for implication and a unary operation
$\neg : 1 \to 1$
for negation as elements of $\Sigma_1$.
A map
\[
  f : m \to 1
\]
is then interpreted as a term with $m$ metavariables, or equivalently as a tree
with $m$ `holes'.
More generally, a map
\[
  f : m \to n
\]
is an $n$-tuple of such terms in a common metavariable context. Under this
viewpoint, maps $0 \to 1$ correspond to object-level terms having no
metavariables, while \emph{generating} maps of type $0 \to 1$ are precisely the
constants of the language.

For example, a composite like
\[
  \imp \cp \neg : 2 \to 1 %
\]
therefore represents the term $\neg (\phi_0 \imp \phi_1)$,
where $\phi_0$ and $\phi_1$ are thought of as \emph{metavariables}.

Notice that this encoding of syntax replaces plain trees by a representation in
which metavariables and their sharing are part of the syntax itself.
For example, the term $\neg (\phi_0 \imp \phi_0)$ - in which $\phi_0$ is shared
- is encoded using the \emph{cartesian} structure of the category by explicitly copying $\phi_0$.

That $\cat{S}$ is \emph{cartesian} is exactly what allows metavariables to be
duplicated or discarded when building terms.
In particular, repeated use of the same metavariable, as for example in the
term $\neg(\phi_0 \imp \phi_0)$, is represented internally by the syntax rather
than imposed as extra bookkeeping on variable names.
This is also why the combinatorial representation of morphisms as cospans of
hypergraphs is natural here: sharing is explicit in the syntax.

\subsection{Categories of proofs}
Equipped with a notion of syntax analogous to formal language,
we now turn our attention to categories of proofs, corresponding to formal
systems.

\begin{definition}\label{def:proof-category}
  A category of proofs $\cat{P}$ over syntax $\cat{S}$
  is the PROP freely generated by operations $\Gamma$ where each operation
  $f : a \to b$ in $\Gamma$ is equipped with a span
  \[
    a \xleftarrow{\;\mathsf{src}(f)\;} m \xrightarrow{\;\mathsf{tgt}(f)\;} b
  \]
  in the syntax category $\cat{S}$, where $m$ is a natural number denoting the
  number of metavariables.
  We call the \emph{arrows} of $\cat{P}$ \emph{derivations}.
\end{definition}

To illustrate this definition, let us return to our earlier example, modus
ponens, which we now assume to be an operation $\mathsf{ax\text{-}mp} : 2 \to 1$ in
$\Gamma$.
We would like to think of it informally as a rule of shape
\[
  A \tensor (A \imp B) \to B,
\]
where $A$ and $B$ are metavariables.
In this case, we have $m = 2$ with the source and target maps given below left
and right, respectively.
\[
  \mathsf{src}(\mathsf{ax\text{-}mp}) :
  (\phi_0,\phi_1) \mapsto (\phi_0,\; \phi_0 \imp \phi_1)
  \qquad\qquad
  \mathsf{tgt}(\mathsf{ax\text{-}mp}) :
  (\phi_0,\phi_1) \mapsto \phi_1
\]
Thus a proof-generating operation carries not only an arity $a \to b$ in the
PROP $\cat{P}$, but also source and target maps in $\cat{S}$ describing its
$a$ hypotheses and $b$ conclusions as tuples of expressions in a common
metavariable context, that is, a span in $\cat{S}$. In the present case,
$\mathsf{src}(\mathsf{ax\text{-}mp}) : 2 \to 2$ encodes the pair of hypotheses
$(\phi_0,\; \phi_0 \imp \phi_1)$, while
$\mathsf{tgt}(\mathsf{ax\text{-}mp}) : 2 \to 1$ encodes the conclusion $\phi_1$.

\begin{remark}
Note that because $\cat{S}$ has cartesian structure, we may alternatively
  regard the $m \to a$ source map as being an $a$-tuple of independent $m \to
  1$ source maps.
  We choose the $m \to a$ representation here so that common subterms built
  from metavariables can be explicitly shared in the source map, but the two
  are equivalent.
\end{remark}

Finally, notice that proof terms in this definition are essentially `untyped':
it is trivial to compose two proof generators whose associated source and target
maps in $\cat{S}$ do not match appropriately.
This is intentional: the syntax exists to be \emph{checked}, rather than to be
a structural representation of a valid proof (as for example in the proof nets
of linear logic~\cite{girard1987linear,curien2005linear2,mellies2009categorical}).

\section{Proof Checking}
\label{section:proof-checking}
We now turn our attention to \emph{checking} proofs.
The section is separated into two parts: the formal semantics of proof
checking, and an algorithm for checking proofs.

In Section \ref{section:syntax-and-proof}, we defined categories of proofs $\cat{P}$
whose arrows were \emph{derivations}.
However, a proof is a derivation of a particular statement: its boundary in
syntax.

\begin{definition}
  Let $\cat{P}$ be a category of proofs over $\cat{S}$,
  let $d : a \to b$ be an arrow in $\cat{P}$,
  and let $s : m \to a$ and $t : m \to b$ be arrows in $\cat{S}$.
  We say that $(s,t)$ is a \emph{type} for $d$.
\end{definition}

An arrow of $\cat{P}$ alone is only a derivation schema.
To regard it as a proof, one must also specify which $a$ hypotheses and which
$b$ conclusions in $\cat{S}$ it is intended to relate.

Types $(s,t)$ play a role analogous to the source and target maps attached to
the generators of $\cat{P}$, but we do not require a derivation to have a
unique type\footnote{
  We suspect that any derivation admits a maximally general type, but do not
  pursue that question here.
}.
This reflects the fact that a single derivation pattern may witness many
different proof instances. For example, the metamath rule $\mathsf{wn}$ states that the
negation of a well-formed formula is well-formed. Its generating source and
target maps send $\phi_0$ to $\mathsf{wff}(\phi_0)$ and
$\mathsf{wff}(\neg(\phi_0))$, respectively. But the same derivation also serves
as a proof of the more specific judgement taking $\mathsf{wff}(\neg(\phi_0))$
to $\mathsf{wff}(\neg(\neg(\phi_0)))$.

\subsection{Valid and Invalid Proofs}
So far, we have not said when a derivation $d : a \to b$ is a \emph{proof} of a
chosen type $(s,t)$.
The point of proof checking is precisely to decide this.

The idea is to map each generator $f$ of a derivation into a partial map
associated to the composite
\[
  \mathsf{src}(f)^- \cp \mathsf{tgt}(f)^+,
\]
where $\mathsf{src}(f)^-$ is a \emph{pattern matcher}, deconstructing the input
syntax to recover metavariables, while $\mathsf{tgt}(f)^+$ constructs the
desired output syntax.
This map is undefined precisely when the input syntax does not conform to the
shape specified by $\mathsf{src}(f)$, meaning that the required hypotheses were
not satisfied.
Since derivations are built inductively from generators, checking a derivation
$d$ satisfies a type $(s, t)$ is done by recursively interpreting generators in
the above way, pre- and post-composing with $s^+$ and $t^-$, respectively, and
evaluating the resulting map.

One can package this as an operational semantics.
From the syntax category $\cat{S}$ one first constructs an auxiliary
symmetric monoidal category which records both term formation and term
matching.

\begin{definition}
  Let $\cat{S}$ be a category of syntax with generators $\Sigma_1$.
  Then $\cat{S}^{\pm}$ is the symmetric monoidal category presented by:
  \begin{itemize}
    \item a special Frobenius structure
      \[
        \mu : 2 \to 1,\qquad \eta : 0 \to 1,\qquad
        \delta : 1 \to 2,\qquad \epsilon : 1 \to 0
      \]
    \item for each generator $g : m \to n$ of $\cat{S}$, a \emph{constructor}
      $g^+ : m \to n$ and naturality equations making $g^+$ a
      comonoid homomorphism with respect to $\delta$ and $\epsilon$;
    \item for each generator $g : m \to n$ of $\cat{S}$, a \emph{matcher}
      $g^- : n \to m$, and dual naturality equations making $g^-$ a
      monoid homomorphism with respect to $\mu$ and $\eta$.
  \end{itemize}
\end{definition}

Thus $\cat{S}^{\pm}$ contains two formal copies of the syntax-generating
operations: one copy for building terms and one for deconstructing them.
By Fox's algebraic presentation~\cite{fox} of cartesian categories, the constructor side
therefore carries a cartesian copy of $\cat{S}$ inside $\cat{S}^{\pm}$, with
copying and discarding represented by $\delta$ and $\epsilon$.
Dually, the matcher side carries an opposite copy of $\cat{S}$, with the
structural maps represented by $\mu$ and $\eta$.
The existence of these two embeddings are formally stated as follows.

\begin{theorem}
  There exist symmetric monoidal functors
  \[
    (-)^+ : \cat{S} \to \cat{S}^{\pm}
    \qquad\text{and}\qquad
    (-)^- : \cat{S}^{\mathrm{op}} \to \cat{S}^{\pm}
  \]
  defined on the generators of $\cat{S}$ as follows:
  \begin{itemize}
    \item for each $g \in \Sigma_1$, one has $(g)^+ = g^+$ and $(g)^- = g^-$;
    \item the cartesian copying and discarding maps of $\cat{S}$ are sent by
      $(-)^+$ to $\delta$ and $\epsilon$, respectively;
    \item the same maps are sent by $(-)^-$ to $\mu$ and $\eta$,
      respectively.
  \end{itemize}
\end{theorem}

\begin{proof}[Proof sketch]
  The assignment on generators is compatible with the structural equations of
  $\cat{S}$ by the imposed naturality relations, and therefore extends
  inductively to composites and tensor products, yielding a symmetric monoidal
  functor $(-)^+ : \cat{S} \to \cat{S}^{\pm}$.
  A similar argument yields $(-)^- : \cat{S}^{\mathrm{op}} \to \cat{S}^{\pm}$.
\end{proof}

The functor $(-)^+$ is the covariant embedding of syntax as term formation.
The functor $(-)^-$ is the contravariant embedding of syntax as pattern
matching.
In particular, if $\Delta : 1 \to 2$ and $! : 1 \to 0$ are the copying and
discarding maps in $\cat{S}$, then $\Delta^+ = \delta$ and $!^+ = \epsilon$,
while $\Delta^- = \mu$ and $!^- = \eta$.
More generally, a syntax map $u : m \to n$ may therefore be read in two
complementary ways: covariantly, as building an $n$-tuple of terms from
$m$ metavariables, and contravariantly, as a matching procedure that tries to
recover $m$ metavariables from an $n$-tuple of terms.

\subsection{Operational Semantics of Derivations}

To make this precise, we interpret arrows of $\cat{S}^{\pm}$ as partial
functions on tuples of rooted syntax trees.
Fix a set $T_{\Sigma_1}$ of trees generated inductively by:
\begin{itemize}
  \item leaves labelled by natural numbers, representing metavariables;
  \item for each generator $g : m \to n$ in $\Sigma_1$, nodes
    $(g,i;t_1,\dots,t_m)$ for $1 \leq i \leq n$ and trees
    $t_1,\dots,t_m \in T_{\Sigma_1}$.
\end{itemize}
Thus a node remembers not only the operation label $g$, but also which output
port $i$ of $g$ it represents.

\begin{definition}
  The \emph{operational semantics} of $\cat{S}^{\pm}$ is the inductively defined
  interpretation of its generators as partial functions on tuples of trees:
  \begin{itemize}
    \item a Frobenius spider $m \to n$ with $m = 0$ creates a fresh leaf and
      returns $n$ copies of it;
    \item a Frobenius spider $m \to n$ with $m > 0$ is defined exactly when all
      $m$ inputs are equal, in which case it returns $n$ copies of that common
      tree;
    \item for a constructor generator $g^+ : m \to n$, the output on
      $(t_1,\dots,t_m)$ is the $n$-tuple
      \[
        ((g,1;t_1,\dots,t_m),\dots,(g,n;t_1,\dots,t_m));
      \]
    \item for a matcher generator $g^- : n \to m$, the input must be an
      $n$-tuple of trees of the form
      \[
        (g,1;t_1,\dots,t_m),\dots,(g,n;t_1,\dots,t_m)
      \]
      with common children $t_1,\dots,t_m$, and the output is then
      $(t_1,\dots,t_m)$; otherwise the function is undefined.
  \end{itemize}
\end{definition}

The meaning of an arbitrary arrow of $\cat{S}^{\pm}$ is then obtained
inductively from this generating data using composition and tensor product.
The special Frobenius structure accounts for equality checking and copying,
while the generators $g^+$ and $g^-$ account for construction and
deconstruction of syntax nodes.
In particular, every syntax map $u : m \to n$ gives rise, via the two
embeddings above, to two arrows
\[
  u^+ : m \to n
  \qquad\text{and}\qquad
  u^- : n \to m
\]
in $\cat{S}^{\pm}$, and hence to two partial functions:
\[
  \llbracket u \rrbracket := \llbracket u^+ \rrbracket
  \qquad\text{and}\qquad
  \llbracket u^- \rrbracket.
\]
The first builds output syntax from metavariable assignments, while the second
attempts to match a tuple of syntax trees against the pattern described by $u$.

\begin{definition}
  The \emph{operational semantics of derivations} is the symmetric monoidal
  functor
  \[
    \llbracket - \rrbracket : \cat{P} \to \mathbf{Pfn}
  \]
  where $\mathbf{Pfn}$ is the category of partial functions
  on tuples of trees, with arrows $m \to n$ given by partial functions
  $T_{\Sigma_1}^m \rightharpoonup T_{\Sigma_1}^n$,
  determined on proof generators $f : a \to b$ by
  \[
    \llbracket f \rrbracket :=
    \llbracket \mathsf{src}(f)^- \cp \mathsf{tgt}(f)^+ \rrbracket.
  \]
\end{definition}

Thus a proof generator $f$ is checked by the partial map
\[
  \llbracket \mathsf{src}(f)^- \cp \mathsf{tgt}(f)^+ \rrbracket,
\]
and the validity criterion for a derivation is simply the following: a
derivation $d$ with proposed boundary $(s,t)$ is valid precisely when
\[
  \llbracket s^+ \rrbracket \cp \llbracket d \rrbracket \cp \llbracket t^- \rrbracket
\]
is defined.

\subsection*{Implementation}
These ideas are realized in the
\href{https://github.com/statusfailed/metacat}{Metacat implementation}, which
provides a surface syntax for defining formal systems together with a checker
for the resulting derivations; the command-line tool can be installed with
\texttt{cargo install metacat-cli}.
Concretely, a derivation is compiled to an open hypergraph, and the
partial-function semantics above is evaluated by a topological traversal of its
operations. The source nodes are first initialized with leaf values,
representing the metavariables supplied at the boundary. One then executes the
operations in topological order, propagating tree values through the
hypergraph. Constructor nodes build syntax trees, matcher nodes inspect and
deconstruct them, and the Frobenius structure accounts for copying and equality
tests. In this way, proof checking becomes an executable evaluation procedure.

\section{First order logic}
\label{section:fol}
As a case study, we now demonstrate how first order logic can be encoded in Metacat.
Our implementation is a direct port from the Metamath~\cite{metamath} standard
library, but we only include a small number of illustrative axioms and propositions.
We begin by defining the relevant fragment of the category of syntax.
For the propositional fragment considered first, the term formers are the
judgement constructor $\mathsf{wff}$ together with the connectives
$\neg$ and $\imp$.
In addition, we include the judgement symbol $\vdash$, which will later be used
to express derivability.
Thus the syntax is generated by the following symbols:
\[
  \mathsf{wff} : 1 \to 1,\qquad
  \vdash : 1 \to 1,\qquad
  \neg : 1 \to 1,\qquad
  \imp : 2 \to 1
\]
Notice that there are two judgement symbols.
The constructor $\mathsf{wff}$ packages a raw formula as a well-formed formula,
while $\vdash$ indicates \emph{provability}.
Although we will not study ill-formed expressions in any detail, the
well-formedness judgement provides a simple first example of how a Metamath
style grammar is represented in the present framework.

\subsection{Well-Formed Formulae}

The first proof generators we introduce are purely \emph{syntactic}.
They express that negation preserves well-formedness, and that implication
takes two well-formed formulae to a well-formed formula.
In Metamath notation these are the rules usually written $\mathsf{wn}$ and
$\mathsf{wi}$.
\[
  \mathsf{wn} : (\phi_0 \mapsto \mathsf{wff}(\phi_0))
  \to (\phi_0 \mapsto \mathsf{wff}(\neg \phi_0))
\]
\[
  \mathsf{wi} : (\phi_0,\phi_1 \mapsto (\mathsf{wff}(\phi_0),\mathsf{wff}(\phi_1)))
  \to (\phi_0,\phi_1 \mapsto \mathsf{wff}(\phi_0 \imp \phi_1))
\]
More explicitly, $\mathsf{wn}$ is a proof generator of arity $1 \to 1$ whose
source map sends a metavariable $\phi_0$ to the single hypothesis
$\mathsf{wff}(\phi_0)$ and whose target map sends $\phi_0$ to the conclusion
$\mathsf{wff}(\neg \phi_0)$.
Similarly, $\mathsf{wi}$ is a proof generator of arity $2 \to 1$ whose source
map sends $(\phi_0,\phi_1)$ to the pair of hypotheses
$(\mathsf{wff}(\phi_0),\mathsf{wff}(\phi_1))$ and whose target map sends the
same metavariables to the conclusion
$\mathsf{wff}(\phi_0 \imp \phi_1)$.

These two generators already illustrate the role of shared metavariables.
The same metavariable assignment must be used consistently across both source
and target maps: the occurrences of $\phi_0$ and $\phi_1$ in the conclusion of
$\mathsf{wi}$ are not new variables, but the same placeholders that appeared in
its hypotheses.
This is exactly the kind of bookkeeping that is represented structurally by the
syntax maps of Section~\ref{section:syntax-and-proof}.

Putting $\mathsf{wn}$ and $\mathsf{wi}$ together allows us to derive that the
negation of an implication is also well-formed.
Concretely, from hypotheses
$\mathsf{wff}(\phi_0)$ and $\mathsf{wff}(\phi_1)$ one first uses
$\mathsf{wi}$ to obtain $\mathsf{wff}(\phi_0 \imp \phi_1)$, and then applies
$\mathsf{wn}$ to conclude
$\mathsf{wff}(\neg(\phi_0 \imp \phi_1))$.
The derivation of this statement is depicted below.
\begin{figure}[H]
\centering
\includegraphics[width=0.8\linewidth]{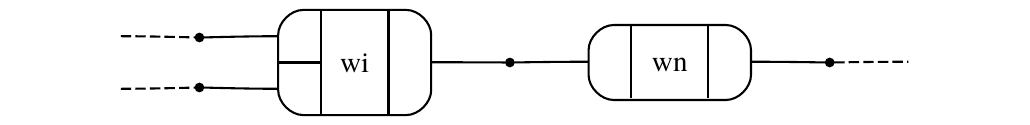}
\end{figure}
This example is deliberately elementary, but it shows the two key ingredients
already at work: syntax maps describe the formulae mentioned by an inference
rule, and proof generators compose exactly as derivations compose.
Subsequent proof rules for propositional and first-order logic will build on
the same pattern.

\subsection{Propositional Logic}

We can now encode some axioms of propositional logic,
which will be sufficient to write the derivation in Figure
\ref{figure:diagram-of-id}.
We now spell out the axioms and inference rules used; namely modus ponens
together with $\mathsf{ax\text{-}1}$ and $\mathsf{ax\text{-}2}$:
\[
  \mathsf{ax\text{-}mp} :
  (\phi_0,\phi_1 \mapsto (\vdash(\phi_0),\vdash(\phi_0 \imp \phi_1)))
  \to
  (\phi_0,\phi_1 \mapsto \vdash(\phi_1))
\]
\[
  \mathsf{ax\text{-}1} :
  (\phi_0,\phi_1 \mapsto (\mathsf{wff}(\phi_0),\mathsf{wff}(\phi_1)))
  \to
  (\phi_0,\phi_1 \mapsto \vdash(\phi_0 \imp (\phi_1 \imp \phi_0)))
\]
\[
  \mathsf{ax\text{-}2} :
  (\phi_0,\phi_1,\phi_2 \mapsto
    (\mathsf{wff}(\phi_0),\mathsf{wff}(\phi_1),\mathsf{wff}(\phi_2)))
  \to
  (\phi_0,\phi_1,\phi_2 \mapsto
    \vdash((\phi_0 \imp (\phi_1 \imp \phi_2))
    \imp ((\phi_0 \imp \phi_1) \imp (\phi_0 \imp \phi_2))))
\]
As before, the source maps record the hypotheses required to form the rule,
while the target maps record the proposition proved by that rule.
Together with the well-formedness generators above, these proof generators are
enough to express the derivation of the identity implication in
Figure~\ref{figure:diagram-of-id}.

\subsection{Quantifiers and First Order Logic}
As a final comment, we must address the encoding of \emph{quantifiers}.
These present a particular challenge as \emph{binders} of variables.
In the present framework, the key point is that variables are not represented
primarily by names, but by their incidence in the underlying syntax object.
Accordingly, a quantified formula should not be thought of as a string together
with a side condition on free and bound names.
Rather, the variable bound by a quantifier is represented by explicit sharing in
the syntax itself: the same metavariable is fed both into the quantifier node
and into the formula being quantified.

For the fragment we consider here, we add a single constructor
\[
  \forall : 2 \to 1,
\]
whose first input represents the bound variable and whose second input
represents the body of the quantified formula.
Diagrammatically, a formula such as $\forall x.\,\phi(x)$ is therefore encoded
by a single shared input wire feeding both the variable slot of $\forall$ and
the occurrences of that variable inside $\phi$.
This is exactly the sort of sharing that is awkward to express cleanly with a
tree syntax plus named substitution, but is immediate in the present setting.

The first proof generator involving quantification is universal
generalization:
\[
  \mathsf{ax\text{-}gen} :
  (x,\phi \mapsto \vdash(\phi))
  \to
  (x,\phi \mapsto \vdash(\forall(x,\phi)))
\]
This should be read with some care.
The metavariable $x$ is present in the common context of the rule, but it need
not occur freely in the premise $\phi$.
That fact is represented structurally rather than by an external side
condition: if the body $\phi$ ignores the variable input, then the syntax map
for the premise simply discards it.

These examples illustrate the main point about binders in Metacat.
The binding behaviour of $\forall$ is not enforced by a special substitution
operation built into proofs; instead, it is already present in the syntax map
describing the rule.
In a fuller treatment one would distinguish variable sorts explicitly, as in
Metamath's treatment of set variables, but we omit that additional structure
here in order to keep the example focused on the categorical representation of
sharing.

\section{Conclusions and Future Work}
\label{section:conclusions}

We have presented a categorical framework for formal systems in which syntax is
represented by a free cartesian PROP and proofs by a free PROP over that
syntax.
In this formulation, inference rules are many-input, many-output operations
equipped with explicit source and target maps in syntax, and proof checking is
separated from proof syntax itself.
This makes it possible to represent invalid derivations syntactically while
giving a precise checking semantics in terms of pattern matching and
construction of syntax.

The symmetric monoidal structure is not an additional decoration on top of an
existing proof formalism.
Rather, it is the basic shape already present in proofs with multiple
hypotheses and conclusions.
Making that structure explicit gives a uniform language for proof assembly and
opens the door to string-diagrammatic and combinatorial representations of
formal reasoning.

We also described a small first-order case study in the style of Metamath,
showing how well-formedness judgements, propositional axioms, and quantifier
rules can be encoded in this framework.

Several directions remain open.
The most immediate is tooling: definitions, abstraction mechanisms, and better
proof authoring support are all needed if Metacat is to be used as a practical
theorem-proving environment.
In addition, we hope that the combinatorial nature of this proof representation
can surface the inherent parallelism in proofs, and by leveraging data-parallel
syntax representations like~\cite{dpafsd} provide automatic parallelisation of
proof checking.
On the theoretical side, it would be interesting to study richer logical
fragments and to make more systematic use of categorical structure to relate
different formal systems by translation and comparison.

\nocite{*}
\bibliographystyle{eptcs}
\bibliography{main}

@book{metamath,
  title={Metamath: A Computer Language for Mathematical Proofs},
  author={Megill, Norman and Wheeler, David A.},
  year={2019},
  publisher={Lulu Press},
  url={https://us.metamath.org/downloads/metamath.pdf}
}

@phdthesis{mm0,
  title={Metamath Zero: From Logic to Proof Assistant to Verified Compilation},
  author={Carneiro, Mario},
  school={Carnegie Mellon University},
  year={2022},
  url={https://digama0.github.io/mm0/thesis.pdf}
}

@article{sdrt1,
  title={String Diagram Rewrite Theory I: Rewriting with Frobenius Structure},
  author={Bonchi, Filippo and Gadducci, Fabio and Kissinger, Aleks and Soboci{\'n}ski, Pawe{\l} and Zanasi, Fabio},
  journal={arXiv preprint arXiv:2012.01847},
  year={2020}
}

@article{sdrt2,
  title={String Diagram Rewrite Theory II: Rewriting with Monoidal Structure},
  author={Bonchi, Filippo and Gadducci, Fabio and Kissinger, Aleks and Soboci{\'n}ski, Pawe{\l} and Zanasi, Fabio},
  journal={arXiv preprint arXiv:2104.14086},
  year={2021}
}

@article{sdrt3,
  title={String Diagram Rewrite Theory III: Confluence and Termination},
  author={Bonchi, Filippo and Gadducci, Fabio and Kissinger, Aleks and Soboci{\'n}ski, Pawe{\l} and Zanasi, Fabio},
  journal={arXiv preprint arXiv:2201.XXXX},
  year={2022}
}

@article{girard1987linear,
  title={Linear Logic},
  author={Girard, Jean-Yves},
  journal={Theoretical Computer Science},
  volume={50},
  pages={1--101},
  url={https://girard.perso.math.cnrs.fr/linear.pdf},
  year={1987}
}

@article{curien2005linear2,
  title={Introduction to Linear Logic and Ludics, Part II},
  author={Curien, Pierre-Louis},
  journal={arXiv preprint cs/0501039},
  year={2005}
}

@book{mellies2009categorical,
  title={Categorical Semantics of Linear Logic},
  author={Melli{\`e}s, Paul-Andr{\'e}},
  year={2009},
  note={Lecture notes}
}

@misc{dafol,
      title={Diagrammatic Algebra of First Order Logic},
      author={Filippo Bonchi and Alessandro Di Giorgio and Nathan Haydon and Pawel Sobocinski},
      year={2024},
      eprint={2401.07055},
      archivePrefix={arXiv},
      primaryClass={cs.LO},
      url={https://arxiv.org/abs/2401.07055},
}

@article{deMoura2021lean4,
  title={Lean 4: A Programming Language and Theorem Prover},
  author={de Moura, Leonardo and Ullrich, Sebastian},
  journal={arXiv preprint arXiv:2104.XXXXX},
  year={2021}
}

@inproceedings{norell2007agda,
  title={Towards a Practical Programming Language Based on Dependent Type Theory},
  author={Norell, Ulf},
  booktitle={International Conference on Types for Proofs and Programs (TYPES 2006)},
  pages={175--189},
  year={2007},
  publisher={Springer}
}

@article{fox,
author  = {Fox, Thomas},
title   = {Coalgebras and Cartesian Categories},
journal = {Communications in Algebra},
volume  = {4},
number  = {7},
pages   = {665--667},
year    = {1976},
month   = {jan},
issn    = {0092-7872},
doi     = {10.1080/00927877608822127}
}

@misc{dpafsd,
      title={Data-Parallel Algorithms for String Diagrams},
      author={Paul Wilson and Fabio Zanasi},
      year={2023},
      eprint={2305.01041},
      archivePrefix={arXiv},
      primaryClass={math.CT},
      url={https://arxiv.org/abs/2305.01041},
}

@misc{chyp,
  author       = {Aleks Kissinger},
  title        = {Chyp: An Interactive Theorem Prover for String Diagrams},
  year         = {2023},
  howpublished = {\url{https://github.com/akissinger/chyp}},
  note         = {Version 0.5.2, Apache-2.0 License}
}

\appendix
\section{Labeled Identity Derivation}
\label{section:appendix-id}
For reference, Figure~\ref{figure:id-labeled} shows a version of the
identity derivation from Figure~\ref{figure:diagram-of-id}, where nodes are labeled with
a text representation of syntax maps using $x_i$ for metavariables.

\begin{figure}[H]
\centering
\includegraphics[width=\linewidth]{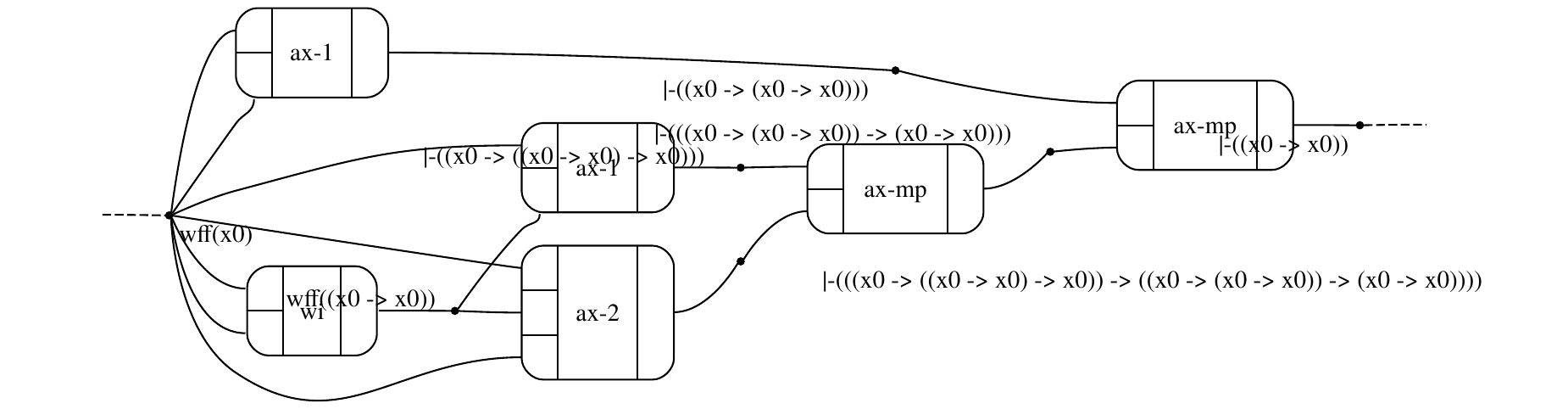}
\caption{A version of the derivation of $\vdash P \imp P$ labeled with tree values at each node}
\label{figure:id-labeled}
\end{figure}

\end{document}